\documentclass[11pt]{amsart}
\usepackage{url, calc}
\usepackage{amsmath,amsxtra,amssymb,latexsym,epsfig,amscd,amsthm,fancybox,epsfig}
\usepackage[mathscr]{eucal}
\usepackage{graphicx}
\usepackage{multicol,xcolor}
\usepackage{epsfig} 
\usepackage{epstopdf}
\usepackage{cases}
\usepackage{subfig}
\usepackage{color}
\usepackage{hyperref}
\newenvironment{eqcite}{
	\begin{equation}
		\begin{aligned}
		}{
		\end{aligned}
	\end{equation}
}
\def\beq{\begin{eqcite}}
	
	\def\eeq{\end{eqcite}}

\newtheorem{thm}{Theorem}[section]
\newtheorem {asp}{Assumption}[section]
\newtheorem{lm}{Lemma}[section]

\newtheorem{pron}{Proposition}[section]
\newtheorem{deff}{Definition}[section]

\newtheorem{lem}{Lemma}[section]
\newtheorem{prop}{Proposition}[section]
\theoremstyle{definition}
\newtheorem{defn}[thm]{Definition}
\newtheorem{hyp}[thm]{Hypotheses}
\theoremstyle{remark}
\newtheorem{rem}{Remark}

\numberwithin{equation}{section}

\DeclareMathOperator{\PI}{P_{inv}}
\DeclareMathOperator{\PE}{P_{erg}}

\DeclareMathOperator{\Conv}{Conv}
\newcommand{\eps}{\varepsilon}

\newcommand{\A}{\mathcal{A}}

\newcommand{\C}{\mathcal{C}}
\newcommand{\CE}{\mathcal{E}}
\newcommand{\CS}{\mathcal{S}}

\newcommand{\M}{\mathcal{M}}
\newcommand{\F}{\mathcal{F}}

\newcommand{\E}{\mathbb{E}}
\newcommand{\BE}{\mathbf{E}}
\newcommand{\BB}{\mathbf{B}}

\newcommand{\BS}{\mathbb{S}}

\newcommand{\BX}{\mathbf{X}}
\newcommand{\bx}{\mathbf{x}}

\newcommand{\BY}{\mathbf{Y}}

\newcommand{\by}{\mathbf{y}}
\newcommand{\BZ}{\mathbf{Z}}

\newcommand{\bz}{\mathbf{z}}

\newcommand{\N}{\mathbb{N}}

\newcommand{\PP}{\mathbb{P}}

\newcommand{\K}{\mathcal{K}}

\newcommand{\R}{\mathbb{R}}

\newcommand{\op}{\mathcal{L}}
\newcommand{\U}{\mathcal{U}}

\numberwithin{equation}{section}

\newcommand{\p}{\mathfrak{p}}
\newcommand{\m}{\mathfrak{m}}

\newcommand{\Lom}{\mathcal{L}}

\newcommand{\bed}{\begin{equation}}
\newcommand{\eed}{\end{equation}}
\newcommand{\bea}{\bed\begin{array}{rl}}
\newcommand{\eea}{\end{array}\eed}
\newcommand{\ad}{&\!\!\!\disp}
\newcommand{\aad}{&\disp}
\newcommand{\barray}{\begin{array}{ll}}
\newcommand{\earray}{\end{array}}
\newcommand{\diag}{{\rm diag}}
\def\disp{\displaystyle}

\newcommand{\1}{\boldsymbol{1}}

\newcommand{\bdelta}{\boldsymbol{\delta}}

\newcommand{\dist}{\mathrm{dist}}

\def\bar{\overline}
\def\hat{\widehat}
\def\a.s{\text{\;a.s.\;}}

\def\al{\alpha}
\newcommand{\wdt}{\widetilde}
\title[Stochastic Extinction]{Stochastic Extinction with Relaxed Boundedness Conditions}

\author[N. Nguyen]{Nhu N. Nguyen }

\address{Department of Mathematics, University of Connecticut, Storrs, CT 06269}
\email{nguyen.nhu@uconn.edu}

\author[D. Nguyen]{Dang H. Nguyen }

\address{Department of Mathematics, University of Alabama, Tuscaloosa, AL
35487}
\email{dangnh.maths@gmail.com}


\begin{document}
\begin{abstract}We study stochastic extinction for a class of Markov processes motivated by models in ecology and epidemiology. Extinction is often characterized by a boundedness condition and a condition on boundary Lyapunov exponents (invasion rates). While the latter is typically sharp, the former is often restrictive and can be improved. Building on the ideas initiated in \cite{benaim2018stochastic}, we develop a streamlined approach that relaxes this boundedness condition and yields concise and accessible criteria for extinction.
	
	In particular, we establish extinction criteria in two settings: with and without a linearly bounded quadratic variation condition. In the first case, our result is comparable to, and slightly improves upon, the main results in \cite{foldes2024stochastic}. In the second case, where the quadratic variation is not linearly bounded, we obtain new extinction results that fall outside the scope of existing frameworks.
	
	Several examples are provided to illustrate the applicability of our results and to highlight situations where previous conditions are not practically verifiable.
\end{abstract}
\maketitle
\tableofcontents
\newpage
\section{Introduction}
Stochastic persistence and extinction are central themes in the study of dynamical systems subject to random fluctuations, with important applications in ecology, epidemiology, and population dynamics. Understanding whether a population persists or goes extinct in the presence of environmental noise is a fundamental question, and significant progress has been made over the past decades in developing mathematical frameworks to address it.

A pioneering contribution in this direction is \cite{benaim2018stochastic}, which laid the foundation for the study of stochastic persistence. In that work, the author introduced a powerful approach based on Lyapunov-type functions and invariant measures to characterize persistence in stochastic systems. The results provide deep insight into how random perturbations influence long-term behavior and have had a lasting impact on the field, inspiring a wide range of subsequent developments, see e.g. \cite{benaim2019random, benaim2019persistence,guillin2019persistence,hening2018coexistence, hening2021general,hening2022classification}. Although \cite{benaim2018stochastic} was intended as the first part of a broader program, with a second part devoted to stochastic extinction, the latter has not been published. Nevertheless, the ideas underlying the treatment of extinction are largely implicit in the framework and have guided several later works addressing extinction or the stability of invariant manifolds.

Building on these ideas, a number of authors have developed criteria for extinction in various classes of stochastic models, often by adapting Lyapunov methods or analyzing the behavior of the process near boundary sets, see the references mentioned above as well as \cite{hening2023stochastic, nguyen2024stability, strickler2021randomly,nguyen2020general,nguyen2020long, nguyen2021stochastic2, tuong2024stochastic}. Typically, extinction is established via two types of conditions: a boundedness condition to manage the behavior near "infinity", and a condition on Lyapunov exponents with respect to invariant probability measures on the boundary (often referred to as invasion rates). The latter condition is often sharp, exhibiting an ``on--off'' threshold behavior, whereas the former remains restrictive in many existing results and leaves room for improvement.

Addressing this issue is also a key motivation of \cite{foldes2024stochastic}, which proposes a general and systematic framework for analyzing extinction phenomena in Markov processes. Their approach, based on a detailed analysis of sample paths and stopping times over the entire state space, provides general conditions ensuring extinction and applies to a broad class of models.

While the framework of \cite{foldes2024stochastic} is mathematically comprehensive  and quite general, it is also technically involved and rather length, what reduces the readiability of the paper. The reliance on global conditions and the need to control the process over the entire state space can make the resulting criteria difficult to verify in concrete applications. Moreover, for certain models of practical interest, these global assumptions may fail to hold, limiting the applicability of the theory.	

The goal of the present paper is to revisit the approach initiated in \cite{benaim2018stochastic} and demonstrate that, with slight but delicate modifications, it can be effectively used to derive extinction results in a simpler and more accessible manner. Our strategy emphasizes the local behavior of the process near a relevant subset (typically an invariant manifold or boundary set) rather than requiring detailed control over the entire state space. This localization allows us to relax some of the global assumptions commonly imposed in the literature, while still obtaining rigorous and meaningful conclusions.

In addition, we establish new results that apply in situations where existing conditions, including those in \cite{foldes2024stochastic}, are not satisfied. In particular, we show that certain key assumptions—such as a linearly bounded quadratic variation—can be replaced by other  conditions without sacrificing the validity of the conclusions. 

Overall, the contributions of this paper are threefold. First, we provide a streamlined and conceptually transparent approach to stochastic extinction, closely aligned with the original ideas of \cite{benaim2018stochastic}. Second, we extend the scope of existing results by identifying conditions under which extinction can be established even when standard assumptions fail. Third, we generalize these results to settings involving multiple Lyapunov functions, which offers a more flexible and user-friendly framework, particularly for high-dimensional models. We also present several examples to demonstrate the applicability and effectiveness of our approach.
%
%
%
%
\section{Main Results}
\subsection{Notation and standing hypotheses}
Let $(M, d)$ be a locally compact Polish (complete and separable) metric space, and
consider a continuous time
Markov process $X(t)$ living in $\M$.
Let
$M_0\subset M$ be a closed invariant subset, called the extinction set. That is
$$X(0) \in M_0 \Leftrightarrow X(t) \in \M_0,\text{ for all }t \geq 0.$$

%

Throughout the paper, we will use the setting and terminology introduced in \cite{benaim2018stochastic}. In particular,  we work with a probability space $(\Omega, \mathcal{F}, \mathbb{P})$, 
a complete right-continuous filtration $(\mathcal{F}_t)_{t \ge 0}$, and a family of càdlàg Markov processes 
$\{(X^x(t))_{t \ge 0}, \; x \in \M\}$ on $(\Omega, \mathcal{F}, (\mathcal{F}_t)_{t \ge 0}, \mathbb{P})$ satisfying

\begin{itemize}
	\item[(i)] For all $x \in \M$, $X(t)$ is an $\M$-valued $\mathcal{F}_t$-measurable random variable, 
	$X_0^x = x$ $\mathbb{P}$-a.s., and $t \mapsto X_t^x$ is càdlàg (i.e., right-continuous with left-hand limits);
	
	\item[(ii)] For each $f \in \mathcal{M}_b(\M)$, the mapping
	\[
	(t,x) \in \mathbb{R}_+ \times \M \mapsto P_t f(x) = \mathbb{E}\big(f(X_t^x)\big)=:\E_x f(X_t)
	\]
	is measurable, and
	\[
	\mathbb{E}\big(f(X_{t+s}^x)\mid \mathcal{F}_t\big) = (P_s f)(X_t^x), 
	\quad \mathbb{P}\text{-a.s.}
	\]
\end{itemize}
We sometimes let $\mathbb{P}_x$ denote the law of $(X^x(t))$ on the Skorokhod space $D(\mathbb{R}_+, \M)$. 
That is,
\[
\mathbb{P}_x(\cdot) = \mathbb{P}\big(\omega \in \Omega : (X_t^x(\omega))_{t \ge 0} \in \cdot \big).
\]
For an invariant subset $A$ of $\M$, we denote by $\PI(A)$ and $\PE(A)$ the families of invariant probability measures and ergodic probability measures respectively.
\begin{deff}\label{def-proper}
	$W:M\to\R$ called a $\M$-proper function if there is a sequence of compact sets $(V_n)_{n\geq 1}$ such that
	$\cup_nV_n = \M$ such that
	$$
	\lim_{n\to\infty}\inf_{x\in \M\setminus V_n}W(x)=\infty.
	$$
\end{deff}

\begin{deff}
	A function $f$ is said to be uniformly integrable w.r.t. a family of probability measures $\C$ if for any $\eps$, there is a $K$ such that
	\begin{equation}\label{eq-10}
		\mu (|f|\1_{|\{f|\geq K\}})\leq \eps \text{ for all }\mu\in \C.
	\end{equation}
	
\end{deff}

We will assume that the following hypotheses holds throughout the section.
\begin{hyp}\label{hyp1}
\begin{enumerate}
	\item \textit{(Invariant decomposition)}
	Let $\M_+\subset\M$ be an invariant set under $(P_t)_{t \ge 0}$.
	There exists a closed set $M_0 \subset \M\setminus\M_+$ called the extinction set of $(P_t)_{t \ge 0}$ which is also invariant under $(P_t)_{t \ge 0}$:
	\[
	\forall t \ge 0, \quad P_t \mathbf{1}_{\M_0} = \mathbf{1}_{\M_0}, \quad P_t \mathbf{1}_{\M_+} = \mathbf{1}_{\M_+}.
	\]
	
	\item \textit{(Feller continuity)}
	For each $f \in C_b(\M)$, the mapping
$
	(t,x) \in \mathbb{R}_+ \times \M \mapsto P_t f(x)
$
	is continuous.
	
	\item \textit{(Tightness on $\M_0$)}
	For each compact set $K\subset \M$, the family $\{X_t^x : x\in K\cap \M_0,\ t\geq 0\}$ is tight.
\end{enumerate}
\end{hyp}
\begin{rem}\begin{enumerate}
		\item There is a slight difference in definition of the extinction set in our setting compared to \cite{benaim2018stochastic,foldes2024stochastic}. We have $\M_0$ is a closed subset of $\M\setminus \M_+$ instead of $\M_0=\M\setminus \M_+$. This setting helps to better characterize the extinction.
\item
	Tightness of the family $\{X_t^x: x \in K \cap \M_0,\; t \ge 0\}$ is often obtained by establishing a uniform bound on the moments of a proper function of $X_t$. 
	Such uniform boundedness is typically derived from a Lyapunov-type condition of the form 
	\[
	\mathcal{L} W(x) \le -\wdt W(x) + \beta,
	\]
	for some constants $\alpha, \beta > 0$ and for proper functions $W, \wdt W$.
		\end{enumerate}
\end{rem}
\begin{lm}
	For any bounded and continuous function $\wdt H:\M\mapsto R$,
	let $\wdt h_m:= \inf\{\mu \wdt H: \mu\in P_{inv}(\M_0)\}$ and $\wdt h_M:= \sup\{\mu \wdt H: \mu\in \PI(\M_0)\}$.
	Then for any $\eps>0$ and any compact set $K\subset \M$ there exists $T_0>0$ such that
	$$
	\wdt h_m-\eps < \frac 1T\int_0^T P_s\wdt H(x)ds< \wdt h_M +\eps \text{ for } x\in K\cap \M_0, T\geq T_0 
	$$
\end{lm}
\begin{proof}
	This is a well-known result coming from the fact that $\{X_t^x, x\in K\cap \M_0, t\geq 0\}$ is tight and any weak limits of the occupation measures $\frac1t\int_0^t P_s(\cdot)ds$ as $t\to\infty$ is an invariant probability measure on $\M_0$. See \cite[Lemma 4.6]{benaim2018stochastic}, \cite[Proposition 1]{schreiber2011persistence}, \cite[Lemma 3.4]{hening2018coexistence} or more details.
\end{proof}
Our conditions will be based on the ``extended" generator and carre du champ operator.

\begin{deff}[Extended Carre du Champ operator]
Let $A\subset M$ be an open invariant set. Define $\mathcal D^{ext}_2(A)$ be set of all continuous $f : A \to \R$ such that there exist continuous functions $\Lom f : A \to \R$ and $\Gamma f : A \to [0, \infty)$ such that $M^f_t(x)$ \begin{equation}\label{eq-mar}
	M^f_t(x) := f(X^x_t) - f(x) - \int_0^t \Lom f(X^x_s) ds
\end{equation}
is a cadlag square integrable martingale and the stochastic process 
\begin{equation}
	(M^f_t(x))^2 - \int_0^t \Gamma f(X^x_s) ds
\end{equation}
is a martingale for
	all $x \in A$.
\end{deff}
\subsection{Extinction under a linearly bounded quadratic variation condition}
\begin{asp}\label{asp2}
	Suppose there exists a function: $U:\M_+\mapsto\R$ 
	such that 
	\begin{enumerate}
		\item $U\in\mathcal D^{ext}_2(\M_+)$.
		\item $\lim_{u\to\infty} \sup\{\dist(x,M_0): U(x)\geq u\}=0$.
 
	\end{enumerate}
\end{asp}
\begin{asp}\label{asp21}
	For the function $U$ in Assumption \ref{asp2}, there exists a continuous function $H:\M\mapsto R$ satisfy that
	\begin{enumerate}

\item $\op U\geq H$ on $\M_+$.
		\item $H$ is
		uniformly integrable w.r.t. $\PI(\M_0)$ and $\inf\{\mu H: \mu\in \PI(\M_0)\}=\lambda>0$
		\item There exists a compact set $C_U\subset \M$ such that 
		$\inf_{\{x\notin C_U\}} H>0$.
						\item $ \sup_{t \geq 1} \frac{1}{t} \int_0^t P_s \Gamma U(x) \, ds \leq c(x)$ where $c(x)$ is bounded in $\M_+\cap C$ for any compact subset $C$ of $\M$.
	\end{enumerate}
\end{asp}

\begin{thm}\label{thm1} 
Assume that	Assumptions  \ref{asp2}, and \ref{asp21} hold. There is $\lambda_0>0$ such that 
for any $\eps\geq 0$ and a compact subset $C$ of $M$, there exists $u_{\eps,C}\geq 0$ satisfying that
	$$
	\PP_x\left\{\lim_{t\to\infty} \dist(X(t),M_0)=0\right\}\geq \PP_x\left\{\liminf_{t\to\infty} \frac{U(X(t))}t\geq \lambda_0>0\right\}\geq 1-\eps, \forall x\in \M_+\cap C\cap \{U\geq u_{\eps,C}\}.
	$$ 

\end{thm}

\begin{rem}
	Let's compare to the main result in  \cite{foldes2024stochastic}, which proves Theorem \ref{thm1} under the following assumption which is a summary of \cite[Assumpstions 3\&4]{foldes2024stochastic}.
\begin{asp}\label{asp-FS}
	There exist proper maps $W, W' : \M \to [1,\infty)$, a function $V : \M \to [0,\infty)$, 
and continuous maps $S, S' : \M \to [0,\infty)$, together with a constant $K>0$, such that:

\begin{itemize}
	\item[(i)] $W \in D_{\mathrm{ext}}^2(\M)$ and $S \in D_{\mathrm{ext}}^+(\M)$;
	
	\item[(ii)] (Drift conditions)
	\[
	L W \le K - W', \qquad L S \le K - S';
	\]
	
	\item[(iii)] (Boundary control via $V$) There exists $V \in D_{\mathrm{ext}}^2(M_+)$ such that:
	\begin{itemize}
		\item[(a)] For any sequence $(x_n) \subset \M_+$, $V(x_n)\to\infty$ implies $d(x_n,M_0)\to 0$;
		\item[(b)] $\op V$ is vanish over  $\{W'\}$, that is, $\lim_{x\to\infty}{L V}{W'}=0 $.
		\item[(c)] $L V$ extends to a continuous function $H : \M \to \mathbb{R}$ and there exists $\lambda>0$ such that
		\[
		\mu H \ge \lambda \quad \text{for all } \mu \in \mathcal{P}_{\mathrm{inv}}(\M_0);
		\]
	\end{itemize}
	
	\item[(iv)] (Carré du champ bounds)
	\begin{equation}\label{cdcb}
	\Gamma W \le K S', \qquad \Gamma V \le K S'.
	\end{equation}
\end{itemize}
	
\end{asp}
	If the assumption \eqref{asp-FS} holds, we can use the function:
	$$
	U=V-cW
	$$
	for a large constant $c$, where $V$ and $W$ satisfies \cite[asumptions 3 \& 4]{foldes2024stochastic}.
	This function will also satisfies Assumptions \ref{asp2} and \ref{asp21}. In particular, \eqref{cdcb} is often the condition used to verify condition (4) of Assumption \ref{asp21}.
	
	However, we do not need that $\op V$ vanish over $|\op W|$ as in \cite[Assumption 3(ii)]{foldes2024stochastic}.
	We just need that $\frac{|\op V|}{1+|\op W|}$ is bounded.
	(Although we will need some similar condition so that $\op V$ is uniformly integrable w.r.t. $P^{inv}(\M_0)$.)
	
Moreover, the proof in \cite{foldes2024stochastic} is based on analyzing sample paths of \(X(t)\) over the entire space \(\mathcal{M}\) at carefully chosen stopping times. Consequently, their approach requires conditions to be imposed globally on \(\mathcal{M}\).
By contrast, our approach focuses on the behavior of \(X(t)\) once it remains in a neighborhood \(M_0\). This approach make it possible to weaken some of the conditions in Assumption~\ref{asp2} and \ref{asp21} by requiring them only in a vicinity of \(M_0\), rather than on the whole space \(\mathcal{M}\).
Since the current proof of Theorem~\ref{thm1} relies on certain global estimates, it would need to be modified under a localized framework to account for the absence of such global bounds. Although these modifications are not technically difficult, we still impose global conditions in this paper to better highlight the main ideas and keep the presentation accessible. The corresponding improvements will be presented in a separate technical note.	
\end{rem}

\begin{proof}[Proof of Theorem \ref{thm1}]
From Assumption \ref{asp21}-condition (1), 
by rescaling  the function $U$ by a new function $\frac{U}{\inf_{\{x\notin C_U\}}\op U}$ if needed, and then rescaling $H$ accordingly, we can assume that $H(x)\geq 1$ for all $x\notin C_U$. We can assume $\lambda<1$ just for simplicity of presentation.

Since $H$ is uniformly integrable w.r.t. $\PI(\M_0)$,
there exists an $\wdt h>0$ such that
	\begin{equation}\label{eq-11}
	\mu (H\wedge \wdt h)\geq 0.9\lambda \text{ for all }\mu\in \PI(\M_0),
	\end{equation}
	where $[H\wedge \wdt h](x)=H(x)\wedge \wdt h$.
	In view of Assumption \ref{asp21}(1), $H(x)\geq -h_1$ for all $x\in \M_0$, for some constant $h_1$.
	Without loss of generality, assume $h_1=\wdt h$, so the function 
	$$\wdt H(x):=H(x)\wedge \wdt h$$ 
	satisfies $|\wdt H(x)|\leq \wdt h$.
	Therefore, as a consequence of \eqref{eq-11}, there exists $T_0\geq 0$ such that 
	$$\frac 1T\int_0^T P_s\wdt H(x)ds\geq 0.8\lambda\,\text{ for }T\geq T_0, x\in C_U\cap M_0.$$
	Define
	\begin{equation}\label{e1.0}
	n_0=	\lceil \wdt h +2\rceil,
	\quad T_1=(n_0-1)T_0,\quad T_2=n_0T_0, \text{ which implies } 
	T_1-\wdt hT_0\geq T_0.
	\end{equation}
Because the process $X$ is Markov-Feller,
there exists a $\wdt u>0$ such that
\begin{equation}\label{e1.2}
	\frac 1T\int_0^T P_s\wdt H(x)ds\geq 0.7\lambda, \,\text{ for any } T_0\leq T\leq T_2 \text{ and } x\in C_U\cap\{U\geq \wdt u\}.
\end{equation}

To proceed, we want to estimate $U(X(t))$ through the following
\begin{eqcite}\label{e1.50}
	U(X(t))=&U(x)+\int_0^t \op U(X(s))ds + M_t^U(x)\\
	\geq& U(x) + \int_0^t \wdt H(X(s))ds + M_t^U(x).
\end{eqcite}

For each integer $n$, we define a stopping time $\xi_n\in[nT_2,(n+1)T_2]$  by \begin{equation}\label{deff-xi-n}\xi_n:=((n+1)T_2)\wedge\inf\{t\geq nT_2: X(t)\in C_U\},\end{equation}
and 
\begin{eqcite}\label{eq1.50}
	\wdt\Delta_n:= &\int_{\xi_n}^{(n+1)T_2} \wdt H(X(s))ds- \E\left[\int_{\xi_n}^{(n+1)T_2} \wdt H(X(s))ds\bigg|\F_{\xi_n}\right],
	\\
	=& \int_{\xi_n}^{(n+1)T_2} \wdt H(X(s))ds- \int_0^{(n+1)T_2-\xi_n}P_s\wdt H(X(\xi_n))ds,
\end{eqcite}
and
\begin{eqcite}\label{eq1.51}
	\wdt G_n:=&\int_{nT_2}^{(n+1)T_2}\wdt H(X(s))ds -\wdt\Delta_n.
\end{eqcite}
Therefore, we have from \eqref{e1.50}, \eqref{eq1.50} and  \eqref{eq1.51} that
\begin{eqcite}\label{e1.5}
	U(X(t))
	\geq& U(x) +\int_{nT_2}^t\wdt H(X(s))ds+\sum_{k=0}^{n-1} \wdt G_k + \sum_{k=0}^{n-1} \wdt \Delta_k + M_t^U(x), \text{ for } t\in[nT_2, (n+1)T_2). 
\end{eqcite}

To proceed, we rewrite $\wdt \Delta_n$ as 
\begin{eqcite}
	\wdt\Delta_n= 
	\int_{\xi_n}^{(n+1)T_2} \wdt H(X(s))ds- \int_0^{(n+1)T_2-\xi_n}P_s\wdt H(X(\xi_n))ds,
\end{eqcite}
and decompose $\wdt G_n$ as 
\begin{eqcite}
	\wdt G_n
	=& \int_{nT_2}^{\xi_n}\wdt H(X(s))ds+\int_0^{(n+1)T_2-\xi_n}P_s\wdt H(X(\xi_n))ds\\
	\geq & (\xi_n-nT_2) + \int_0^{(n+1)T_2-\xi_n}P_s\wdt H(X(\xi_n))ds, \text{ since } \wdt H(x)\geq 1 \text{ if } x\notin C_U. 
\end{eqcite}

Now, consider 3 events whose union contains the event $\{X(\xi_n)\notin C_U\cap \{U<\wdt u\}\}$.
\begin{itemize}
	\item 
	If $\xi_n=(n+1)T_2$, then $\wdt G_n\geq T_2$.
	
	\item If $X(\xi_n)\in C_U\cap \{U\geq \wdt u\}$ and $\xi_n\leq nT_2+T_1$, then \eqref{e1.2} implies
	$$\int_0^{(n+1)T_2-\xi_n}P_s\wdt H(X(\xi_n))ds\geq ((n+1)T_2-\xi_n)0.7\lambda\geq 0.7\lambda(T_2-T_1)=0.7\lambda T_0,$$
	which is followed by $\wdt G_n\geq 0.7\lambda T_0.$ 
	
	\item If $X(\xi_n)\in C_U\cap \{U\geq \wdt u\}$ and $\xi_n> nT_2+T_1$, then
	$$
	\wdt G_n\geq (\xi_n-nT_2)-\wdt h((n+1)T_2-\xi_n)\geq T_1-\wdt hT_0\geq T_0 \text{ (due to \eqref{e1.0})}.
	$$
\end{itemize}
Therefore, in summary, we have
\begin{eqcite}\label{e1.4}
	\wdt G_n\geq 0.7\lambda T_0 \text{ if } X(\xi_n)\notin C_U\cap \{U<\wdt u\}.
\end{eqcite}

Next, we have the following auxiliary lemma, whose proof is given in the appendix.
\begin{lm}\label{lm1-thm1}
	Consider 2 events
	$$
	A_{M}^R=\left\{|M_t^U(x)|\leq R+\frac{0.1\lambda t}{n_0},\text{ for all } t\geq 0\right\},
	$$
	and
	$$
	B_{M}^R=\left\{\sum_{k=0}^{n-1}|\wdt\Delta_k|\leq R+\frac{0.1\lambda n T_2}{n_0},\text{ for all } n\geq 0\right\}.
	$$
	For any $\eps>0$ and a compact set $C$, there exists $R=R(\eps,C)>0$ such that
	$$
	\PP_x(A_{M}^R)\geq 1-\frac\eps2 \text{ and } \PP_x(B_{M}^R)\geq 1-\frac\eps2, \text{ for any } x\in C\cap \M_+.
	$$
\end{lm}
Let $R$ be as in Lemma \ref{lm1-thm1} corresponding to $C_U$ and $\eps$. It is noted that $R$ depends only on $C_U$ and $\eps$, and is independent of initial value $x\in C\cap\M_+$.	
	Let \begin{equation}\label{e1.8}
		u_{\eps, C}:=\wdt u+\wdt hT_2+2R+0.2\lambda T_0.
	\end{equation} 
	
	We will show that for any initial condition $ x\in C\cap\M_+$ satisfying $U(x)>u_{\eps, C},$ we have:
	\begin{eqcite}\label{e1.6}
		X(\xi_n)\notin C_U\cap\{U<\wdt u\},\; \forall\, n\in\N \text{ for almost all } \omega \in  A_{M,\eps}^R\cup B_{M,\eps}^R. 
	\end{eqcite}
	Indeed, if \eqref{e1.6} is not true, we can find $D\subset A_{M,\eps}^R\cup B_{M,\eps}^R$ such that $\PP_x(D)>0$ and for $\omega\in D$, we have
	$$
	X(\xi_n)\in C_U\cap\{U<\wdt u\} \text{ for some } n\in\N.
	$$
	Define 
	$$
	\wdt m=\inf\left\{n\in\N:X(\xi_n)\in C_U\cap\{U<\wdt u\}\right\},
	$$
	which is finite in $D$.
	In view of \eqref{e1.4}, we have $\wdt G_k\leq 0.7\lambda T_0$, for all $k\leq \wdt m$. Combining this fact and \eqref{e1.5} leads to that for $t\in[\wdt mT_2, (\wdt m+1)T_2]$
	\begin{eqcite}\label{e1.7}
		U(X(t))
		\geq& U(x) +\int_{\wdt mT_2}^t\wdt H(X(s))ds+\sum_{k=0}^{\wdt m-1} \wdt G_k + \sum_{k=0}^{\wdt m-1}v \wdt \Delta_k + M_t^U(x)
		\\
		>& u_{\eps,C}-\wdt hT_2+0.7\lambda T_0 \wdt m - 2R - 0.2\lambda\frac{t}{n_0}
		\\
		>& u_{\eps,C}-\wdt hT_2+0.5\lambda T_0\wdt m -2R-0.2\lambda T_0 
		\\
		>& \wdt u;
	\end{eqcite}
	where the second last inequality follows the estimate
	$$
	0.2\lambda\frac{t}{n_0}\leq 0.2\lambda\frac{(\wdt m+1)T_2}{n_0}\leq 0.2\lambda T_0\wdt m  +0.2\lambda T_0, \text{ because of } T_2=n_0T_0.
	$$
	Because
	\eqref{e1.7} contradicts the definition and finiteness of $m$ based on the contradiction assumption,
\eqref{e1.6} holds true.
As a consequence, $U(X(t))>\wdt u$ for all $t\geq 0$ for almost all $\omega \in  A_{M,\eps}^R\cup B_{M,\eps}^R$. Thus, by dividing by $t$ and letting $t\to\infty$ in \eqref{e1.5}, we obtain that, 
	$$
	\liminf_{t\to\infty} \frac{U(X(t))}t\geq \frac{0.7\lambda}{n_0}=:\lambda_0,\text{ for almost all }\omega \in  A_{M,\eps}^R\cup B_{M,\eps}^R.
	$$
	By noting that $$\PP_x\left(A_{M,\eps}^R\cup B_{M,\eps}^R\right)\geq 1-\eps,$$
	the proof is complete.
\end{proof}
\begin{rem}
	The main idea is to decompose the integrals $\int_0^t \wdt H(X(s))\,ds$ along carefully chosen sub-intervals.
	More precisely, for each integer $n$, we define a stopping time $\xi_n \in [nT_2,(n+1)T_2]$ by \eqref{deff-xi-n}. 
	This represents a subtle but important modification of the decomposition used in \cite{benaim2018stochastic}.
	Instead of writing
	\[
	\int_{nT_2}^{(n+1)T_2}\wdt H(X(s))\,ds
	= \int_{nT_2}^{(n+1)T_2} P_s \wdt H\big(X(nT_2)\big)\,ds
	+ \left(
	\int_{nT_2}^{(n+1)T_2}\wdt H(X(s))\,ds
	- \int_{nT_2}^{(n+1)T_2} P_s \wdt H\big(X(nT_2)\big)\,ds
	\right),
	\]
	we wait until time $\xi_n$ and then decompose the process into its conditional expectation given $X(\xi_n)$ and a martingale increment.
	
	The key advantage is that $\op U \geq 1$ holds uniformly on $[nT_2,\xi_n]$, so no decomposition is needed before $\xi_n$. 
	This simplifies the estimates and removes several unnecessary technical conditions.
\end{rem}
\subsection{Extinction when the quadratic variation is not linearly bounded}
A key condition for Theorem~\ref{thm1} is the Carré du champ bound (see~(3) in Assumption~\ref{asp2}), which is also imposed in existing results (see, e.g.,~(iv) of Assumption~\ref{asp-FS}). However, this condition is not satisfied in some models (see examples in the next section).
The following theorem addresses this issue.

\begin{asp}\label{asp4}
	There exists a proper function $W: \M\mapsto\R$ (see Definition \ref{def-proper}) such that
	\begin{enumerate}
		\item $W\in\mathcal D^{ext}_2(\M)$, $M_t^W(x)$ is continuous.
		\item $\op W\leq K_{W}-\gamma_W W$ and $\Gamma W\leq K_{W} W^2$.
		\item $\frac{|H|^{p_0}+(\Gamma U)^{p_0/2}}{1+W}$ is bounded for some $p_0>1$.  
		\item $\op [(W)^{2}]\leq k_W (W)^{2}$. 
	\end{enumerate}
	
\end{asp}	
\begin{rem}
We assume that \(M_t^W(x)\) is a continuous martingale so as to avoid handling jump contributions. In the discontinuous case, Theorem~\ref{thm2} continues to hold under appropriate conditions ensuring sufficient control of the jump magnitudes.

Under Assumption~\ref{asp4}, the quadratic variation of the Lyapunov function may grow exponentially fast. However, the proof shows that the associated martingale admits a finite \(p\)-th moment with at most linear growth for some \(p>1\). 
\end{rem}
\begin{thm}\label{thm2} 
	Assume that Assumptions  \ref{asp2} and \ref{asp4} hold. There is $\lambda_0>0$ such that for any $\eps\geq 0$ and a compact subset $C$ of $\M$, there exists $u_{\eps,C}\geq 0$ satisfying that
	$$
	\PP_x\left\{\liminf_{t\to\infty} \frac{U_t}t\geq \lambda_0>0\right\}\geq 1-\eps \text{ for all } x\in \M_+\cap C\cap \{U\geq u_{\eps,C}\}.
	$$ 
\end{thm}
Next, we present the following result, whose proof is given in the appendix. It can be viewed as a modification of \cite[Theorem~2.19]{hall2014martingale}.
\begin{pron}\label{pron1}
	Let $\{X_n\}$ be a sequence of random variables  and $\F_n$ adapted. If
	$$\sup_{n}\E X_n^p <\infty, p>1$$
	then for any $\eps>0$, $\delta>0$,  there exists $K>0$ such that 
	$$
	\PP\left\{\sum_{i=1}^n\left[X_i-\E[X_i\big|F_n]\right]\leq K+\delta n,\;\forall n\right\}\geq 1-\eps.
	$$	
\end{pron}
\begin{proof}[Proof of Theorem \ref{thm2}]
We will only consider the case when $p_0<2$ because if $p_0\geq 2$, we can use Theorem \ref{thm1}.
Moreover, in this following, excepting for $K_W$ it is a fixed constant given in Assumption \ref{asp4}(2), $K$ will represent a generic constant, which is different in each its appearance. We will indicate which $K$ depends on, if it is necessary.
	
From Assumption \ref{asp4} - condition (2), 
by applying It\^o's formula for $e^{\gamma_Wt}W(X(t))$ we obtain that
\begin{equation}\label{bound-eW}
	\E_x e^{\gamma_Wt}W(X(t))\leq W(x)+\frac{K_W(e^{\gamma_Wt}-1)}{\gamma_W}.
\end{equation}
Therefore, for each compact $C\in \M$, 
\begin{equation}\label{eq-boundEW}
\sup_{x\in C, t\geq 0} \{P_tW(x)=\E_x W(X(t))\}=K_C<\infty.
\end{equation}

Similarly, thanks to Assumption \ref{asp4} - condition (4), by applying It\^o's formula  for $e^{-k_Wt}W^2(X(t))$, we obtain that
\begin{equation*}
	\E_x e^{-k_Wt}W^2(X(t))\leq W^2(x).
\end{equation*}
As a consequence,
\begin{equation}\label{eq-EW2}
\E_xW^2(X(t))=P_tW^2(x)\leq e^{k_Wt}W^2(x).
\end{equation}
Thus, from Markov's property \& Jensen's inequality:
\begin{equation}\label{eq-EW2-2}
\E \left(\int_{t_1}^{t_2} W^{2}(X(s))ds\right)^{1/2} \leq\frac 1{\sqrt {k_W}} e^{\frac{k_W}2(t_2-t_1)}\E W(X(t_1)).
\end{equation}

In view of Assumption \ref{asp4}- condition (2) and It\^o's formula, for $q=2/(1+p_0)<1$, we have
\begin{equation}\label{eq-W1q-bound}
[W(X(t))]^q\leq [W(x)]^q+ qK_{W}t - q\gamma_W \int_0^t [W(X(s))]^q ds + M_t^{W^q},
\end{equation}
where  $M_t^{W^q}$ is defined as \eqref{eq-mar} with the function $f$ being $W^q$.
Moreover, the quadratic variation $\langle M^{W^q}\rangle_t$ of $M_t^{W^q}$ satisfies that
\begin{equation}\label{eq-qv}
\langle M^{W^q}\rangle_t=\int_0^t [\Gamma (W)^q](X(s))ds=q^2\int_0^t \frac{[\Gamma W](X(s))}{[W(X(s))]^{2-2q}}ds\leq K \int_0^t [W(X(s))]^{2q} ds,
\end{equation}
for some constant $K$, due to Assumption \ref{asp4}(3).

For any $T$ and integer $n$, because of Burkholder-Davis-Gundy inequality and \eqref{eq-qv}, we have
we have
\begin{equation}\label{eq-thm12-1}
\E_x \left|M_{(n+1)T}^{[W]^q}-M_{nT}^{[W]^q}\right|^{1/q}
\leq
K\E_x\Big(\int_{nT}^{(n+1)T}[W(X(s))]^{2q} ds\Big)^{1/2q}.
\end{equation}
Moreover, applying Holder inequality and \eqref{eq-EW2-2}, one has
\begin{equation}\label{eq-thm12-2}
	\E_x\Big(\int_{nT}^{(n+1)T}[W(X(s))]^{2q}ds\Big)^{1/2q}\leq \E_x\Big(\int_{nT}^{(n+1)T}[W(X(s))]^{2}ds\Big)^{1/2}\leq K  e^{\frac{k_WT}2}\E_x W(X(nT)).
\end{equation}
Combining \eqref{eq-thm12-1}, \eqref{eq-thm12-2}, and \eqref{eq-boundEW}, we get that for any compact set $C$ there is an $M_{C,q}$ depending on $C$ only such that
\begin{equation}
 \E_x \left|M_{(n+1)T}^{[W]^q}-M_{nT}^{[W]^q}\right|^{1/q}\leq M_{C,q},\text{ for all } x\in C, \;T\in [0,1].  
\end{equation}

Applying Proposition \ref{pron1} for the sequence $|M_{(n+1)T}^{[W]^q}-M_{nT}^{[W]^q}|$, for any compact set $C$,
there is a $K_1\geq 0$, which depends on  $\eps,\lambda,M_{C,q}$, such that
\begin{equation}\label{e3.11} 
\PP_x\left\{M_{nT}^{[W]^q}\leq K_1+0.1\lambda nT,\;\forall n\in\N\right\}\geq 1-\eps,\text{ for all } x\in C.
\end{equation}

In view of Lemma \ref{lm-a3}, for any $\eps>0$, we can find an $h_q>0$ such that
$$
\sum_{i=1}^\infty \frac{M_{C,q}}{(h_q+0.05\lambda i)^{1/q}}\leq \eps.
$$
By Markov's inequality, 
$$
\PP_x\left\{\sup_{s\in[0,T_2]}\left|M_{kT_2+s}^{[W]^q}-M_{kT_2}^{[W]^q}\right|\leq h_q+0.05\lambda k\right\}\geq 1-\frac{M_{C,q}}{(h_q+0.05\lambda k)^{1/q}}
$$
which leads to
\begin{equation}\label{e3.12}
\PP_x\left\{\sup_{s\in[0,T_2]}\left|M_{kT_2+s}^{[W]^q}-M_{kT_2}^{[W]^q}\right|\leq h_q+0.05k, \text{ for all } k\in\N\right\}\geq 1-\sum_{k=1}^\infty\frac{M_{C,q}}{(h_q+0.05\lambda i)^{1/q}}\geq 1-\eps.
\end{equation}
Combining \eqref{e3.11} and \eqref{e3.12}, we have  for any compact set $C$, with  $K_{2}:=K_1+h_q$ depending only on $C$ and $\eps$ that
%
%
%
\begin{equation}\label{eq-MtQ-bound}
\PP_x\left\{M_{t}^{[W]^q}\leq K_{2}+0.1\lambda t,\;\forall t\geq 0\right\}\geq 1-\eps,\text{ for all } x\in C.
\end{equation}
Combining \eqref{eq-W1q-bound} and \eqref{eq-MtQ-bound} leads to that
\begin{equation}\label{eq-Wq}
\PP_x\left\{\int_0^t [W(X(s))]^q ds \leq K_2+(0.1\lambda+qK_{W}) t,\;\forall t\geq 0\right\}\geq 1-\eps,\text{ for all } x\in C.
\end{equation}

Now, we have again that
\begin{eqcite}\label{e3.50}
	U(X(t))\geq&U(x)+\int_0^t H(X(s))ds + M_t^U(x).
\end{eqcite}
In the next step, 
we aim to write function $H\geq \wdt H_1 - \wdt H_2$, where $\wdt H_1, \wdt H_2$ are chosen as follows.
Because of \eqref{eq-EW2} and Assumption \ref{asp4}(4), one can have that $H$ is uniformly integrable w.r.t $\PI(\M_0)$.
In view of \eqref{bound-eW}, we have
$$
\lim_{t\to\infty}\E_x W(X(t) \leq K_W\gamma_W^{-1} \text{ for all } x\in \M,
$$
which leads to 
$$
\mu W\leq K_W\gamma_W^{-1} \text{ for all } \mu\in \PI(\M).
$$
Then, together with Assumption \ref{asp4} - condition (3),  it is easy to find a compact set  $C_1$ of $\M$, a constant $K_3>0$ such that
\begin{equation}\label{h-bound1}
|H(x)|+1\leq K_3(1+[W(x)]^{1/p_0}) \text{ and } \frac{0.1\lambda+qK_W}{[\inf_{x\notin  C_1}W(x)]^{q-\frac 1{p_0}}} \leq 0.1\lambda \text{ for } x\notin C_1
\end{equation} 
and
\begin{equation}\label{h-bound2}
K_3\mu\left(\1_{\{x\notin C_1\}}(1+W^{1/p_0})\right)\leq 0.1\lambda  \text{ for all } \mu\in \PI(\M).
\end{equation}
Since $W$ is a proper function, there is a $C_2\supset C_1$ being a compact subset of $\M$ so that $\inf_{x\notin C_2} W(x)> \sup_{x\in C_1} W(x)$.
It is easy to construct a continuous function $g:\M\mapsto [0,1]$ such that
$g(x)=0$ if $x\in C_1$ and $g(x)=1$ if $x\notin C_2$.
Let $\wdt H_1= g + (1-g)H$ and $\wdt H_2= \wdt H_1-H$.
Then, 
$$|\wdt H_1(x)|\leq |H(x)|+1 \text{ and } |\wdt H_2(x)|\leq |H(x)|+1.$$

We also have that
\begin{equation}
	\wdt H_1(x)=1 \text{ if } x\notin C_2 \text{ and } \wdt H_2(x)=0 \text { if } x\in C_1.
\end{equation}
Due to \eqref{h-bound1} and \eqref{h-bound2}, $\wdt H_1,\wdt H_2$ are uniformly integrable  for $\mu\in \PI(\M_0)$ and
\begin{equation}\label{muh2}
	\mu |\wdt H_2| \leq \int_{\M\setminus C_1} (|H(x)|+1)\mu(dx)\leq 0.1\lambda 
\end{equation}

From \eqref{e3.50}, we obtain that
\begin{eqcite}\label{e3.5}
	U(X(t))\geq U(x) + \int_0^t \wdt H_1(X(s))ds - \int_0^t \wdt H_2(X(s))ds+M_t^U(x).
\end{eqcite}

We have that on the event $\big\{\int_0^t [W(X(s))]^q ds \leq K_{2}+(0.1\lambda+qK_{W}) t,\;\forall t\geq 0\big\}$
\begin{align*}
	\limsup_{t\to\infty}\frac 1t\int_0^t |\wdt H_2(X(s))| ds&=\limsup_{t\to\infty}\frac{1}{t} \int_0^t \mathbf{1}_{\{X(s)\notin  C_1\}}(|H(X(s))|+1)ds\\
	&\leq \limsup_{t\to\infty}\frac{1}{t} \int_0^t \mathbf{1}_{\{X(s)\notin  C_1\}}(W(X(s)))^{\frac 1{p_0}}ds\\
	&\leq \limsup_{t\to\infty}\frac{1}{t} \int_0^t \frac{(W(X(s)))^{q}}{[\inf_{x\notin  C_1}W(x)]^{q-\frac 1{p_0}}}ds \\
	&\leq \frac{0.1\lambda+qK_W}{[\inf_{x\notin  C_1}W(x)]^{q-\frac 1{p_0}}} < 0.1\lambda.
\end{align*}
As a result, one has
\begin{equation}\label{averagewdth2}
\PP_x\left\{\limsup_{t\to\infty}\frac 1t\int_0^t \wdt H_2(X(s)) ds \leq 0.1\lambda\right\}\geq 1-\eps.
\end{equation}

By a similar process of getting \eqref{eq-MtQ-bound}, we can obtain that 
\begin{equation}\label{averageMsU} 
\PP_x\left\{\limsup_{t\to\infty}\frac 1t\int_0^t M_s^U(x) ds \leq 0.1\lambda\right\}\geq 1-\eps.
\end{equation}


Now, note that $\wdt H_1(x)\geq 1$ for all $x\notin C_2$ and
$\mu\wdt H_1(x)=\mu\wdt H +\mu\wdt H_2\geq 0.9\lambda \text{ for } \mu\in \PI(\M_0)$, which is due to \eqref{muh2}. With these properties, to estimate $\frac 1t\int_0^t\wdt H_1(X(s))ds$, we 
break down into $\wdt G_k$ and $\wdt \Delta_k$ like \eqref{eq1.50} and \eqref{eq1.51} and
we can use the arguments in Theorem \ref{thm1} and \eqref{averagewdth2} and \eqref{averageMsU} to obtain that
\begin{equation} 
\PP_x\left\{\limsup_{t\to\infty}\frac 1t\int_0^t \wdt H_1(X(s)) ds \geq 0.7\lambda\right\}\geq 1-3\eps \text{ if } U(x)\geq  u_{\eps,C}
\end{equation}
for a sufficiently large $u_{\eps,C}$.
The proof is complete.
\end{proof}

\subsection{Extention to multiple Lyapunov functions}

\begin{asp}\label{asp5}
		Suppose there exists  functions: $U_i: \M_+\mapsto\R$, continuous functions: $H_i: \M\mapsto\R, i=1,\cdots, n$ satisfying
		\begin{enumerate}
			\item $U_i\in\mathcal D^{ext}_2(\M_+)$.
			\item $\lim_{u\to\infty} \sup\{\dist(x,M_0): U(x):=\min_{1\leq i\leq n}U_i(x)\geq u\}=0$.
			\item $ \sup_{t \geq 1} \frac{1}{t} \int_0^t P_s \Gamma U_i(x) \, ds \leq c(x)$ where $c(x)$ is bounded in $\M_+\cap C$ for any compact subset $C$ of $M$. 
			\item $\op U_i\geq H_i$ on $\M_+$.
			\item $H_i$ is
			uniformly integrable w.r.t. $\PI(\M_0)$ and $\inf\left\{\mu H_i: \mu\in \PI(\M_0), i\in\{1,\cdots,n\}\right\}=\lambda>0$
				\item There exists a compact set $C_U\subset \M$ such that 
				$\inf_{\{x\notin C_U\}} H_i>0$.
				
			\end{enumerate}

\end{asp}
\begin{thm}\label{thm5} 
	Assume that	Assumption \ref{asp5} holds. There is $\lambda_0>0$ such that 
	for any $\eps\geq 0$ and a compact subset $C$ of $\M$, there exists $u_{\eps,C}\geq 0$ satisfying that
	$$
	\PP_x\left\{\lim_{t\to\infty} \dist(X(t),\M_0)=0\right\}\geq \PP_x\left\{\liminf_{t\to\infty} \frac{U(X(t))}t\geq \lambda_0>0\right\}\geq 1-\eps, \forall x\in \M_+\cap C\cap \{U\geq u_{\eps,C}\}.
	$$ 
\end{thm}
\begin{proof}[Sketch of Proof of Theorem \ref{thm5}]
	The proof is largely analogous to that of Theorem \ref{thm1}. 
	
	We can find $\widetilde h$ such that 
	$|\widetilde H_i(x):=H(x)\wedge \widetilde h|\leq \widetilde h$ for all $x\in\M$, and 
	\[
	\frac{1}{T}\int_0^T P_s \widetilde H_i(x)\,ds \geq 0.8\lambda,
	\quad \text{for } T\geq T_0,\ x\in C_U\cap \M_0.
	\]
	
	Then, we define $T_0,T_2$ as in \eqref{e1.0}, and choose $\widetilde u$ such that \eqref{e1.2} holds with $\widetilde H$ replaced by $\widetilde H_i$ for all $i\in\{1,\dots,m\}$.
	
	With $\xi_n$ defined in \eqref{deff-xi-n}, and $\widetilde \Delta_{i,k}$ and $\widetilde G_{i,k}$ defined in \eqref{eq1.50} and \eqref{eq1.51} (with $\widetilde H$ replaced by $\widetilde H_i$), we obtain the following inequality, analogous to \eqref{e1.5}:
	\begin{eqcite}\label{e5.5}
		U_i(X(t))
		\geq U_i(x) + \int_{nT_2}^t \widetilde H_i(X(s))\,ds
		+ \sum_{k=0}^n \widetilde G_{i,k} 
		+ \sum_{k=0}^n \widetilde \Delta_{i,k} 
		+ M_t^{U_i}(x),
		\quad t\in[nT_2,(n+1)T_2).
	\end{eqcite}
	
	Similar to Lemma \ref{lm1-thm1}, for any $\varepsilon>0$ and any compact set $C\subset\M$, there exists $R>0$ such that
	\begin{equation}\label{e5.6}
		\PP_{x}\Bigg(
		\Big\{|M_t^{U_i}(x)| \leq R+\tfrac{0.1\lambda t}{n_0},\ \forall t\geq 0\Big\}
		\cap 
		\Big\{\Delta_{i,n}\leq R+\tfrac{0.1\lambda n}{n_0},\ \forall n \geq 0\Big\}
		\Bigg)\geq 1-\varepsilon.
	\end{equation}
	
	Using \eqref{e5.5} and \eqref{e5.6}, and mimicking the remainder of the proof of Theorem \ref{thm1} after \eqref{e1.5}, we obtain the desired conclusion.
\end{proof}
Similarly, we can extend Theorem \ref{thm2} to
\begin{thm}\label{thm6}
	The conclusion of Theorem \ref{thm5} remains valid if part (3) of Assumption \ref{asp5} is replaced by the following.
	
	There exists a proper function $W:\M\to\R$ such that:
	\begin{enumerate}
		\item $W\in\mathcal D^{ext}_2(\M)$.
		
		\item $\mathcal{L} W \leq K_W - \gamma_W W$ and $\Gamma W \leq K_W W^2$.
		
		\item The function 
$
		\frac{|H_i|^{p_0} + (\Gamma U_i)^{p_0/2}}{1+W}
$
		is bounded for some $p_0>1$.
		
		\item $\mathcal{L}(W^2) \leq k_W W^2$.
	\end{enumerate}
\end{thm}
\section{Applications}

\subsection{Extinction of stochastic Kolmogorov systems with regime-switching}
We revisit a stochastic Kolmogorov system. A comprehensive study of this system is provided in \cite{hening2018coexistence} for stochastic differential equations (SDEs), and later presented in \cite{nguyen2025hybrid} in the setting of  SDEs with regime switching. The system is the couple \eqref{kol} and \eqref{c14-eq:tran} below describing the dynamics of $n$ interacting populations $(X_1(t),\dots,X_n(t))$.
\begin{equation}\label{kol}
	dX_i(t)=X_i(t) f_i(\BZ(t))dt+X_i(t)g_i(\BZ(t))dE_i(t), ~i=1,\dots,n,
\end{equation}
where
$\BZ(\cdot)= (\BX(\cdot),  \alpha(\cdot))$, $\BX(t)= (X_1(t),\ldots, X_n(t))$, and
$ \alpha(t)$ is an irreducible Markov chain on a finite state space $\mathcal S=\{1,\cdots,m_0\}$ with generator $Q=(q_{ij})_{m_0\times m_0}$,
that is,
\begin{equation}\label{c14-eq:tran}\begin{array}{ll}
		&\disp \PP\{ \alpha(t+\Delta)=j| \alpha(t)=i,
		 \alpha(s), s\leq t\}=q_{ij}\Delta+o(\Delta) \text{ if } i\ne j \
		\hbox{ and }\\
		&\disp \PP\{ \alpha(t+\Delta)=i| \alpha(t)=i,
		 \alpha(s), s\leq t\}=1+q_{ii}\Delta+o(\Delta).\end{array}
\end{equation}
We use the following notation for simplicity
$\R^n_+=[0,\infty)^n, \R^{n,\circ}_+=(0,\infty)^n, \partial\R^n_+=\R^n_+\setminus \R^{n,\circ}_+$, ${\mathbb S}=\R^n_+\times\mathcal S, {\mathbb S}^\circ=\R^{n,\circ}_+\times\mathcal S,\partial{\mathbb S}=\partial\R^n_+\times\mathcal S$, $\bx=(x_1,\dots,x_n)$,
$\bz=(\bx,\alpha )$.
Denote
$\BE(t)=(E_1(t),\dots, E_n(t))'=\Gamma'\BB(t)$
$\Gamma$ is an $n\times n$ matrix such that
$\Gamma'\Gamma=\Sigma=(\sigma_{ij})_{n\times n}$
and $\BB(t)=(B_1(t),\dots, B_n(t))$ is a vector of independent standard Brownian motions adapted to the filtration $\{\F_t\}_{t\geq 0}$.  

We  write $\op$,  the generator of the process $(\BZ(t)):=((X_1(t),\dots,X_n(t), \alpha(t))$ as
\bea
[\op F](\bz ) \ad = \sum_i x_if_i(\bz )\frac{\partial F}{\partial x_i}(\bz ) + \frac{1}{2}\sum_{i,j}\sigma_{ij}x_ix_jg_i(\bz )g_j(\bz )\frac{\partial^2 F}{\partial x_i \partial x_j}(\bz )\\
\aad  \ \hfill +\sum_{j\in\mathcal S} q_{mj}F(\bx,j).
\eea
and use the norm $\|\bx\|=\sum_{i=1}^n |x_i|$ in $\R^n$.
The following assumption will guarantee the existence and uniqueness of a strong solution to \eqref{kol} and \eqref{c14-eq:tran} for a given initial value and the solution process is a Markov-Feller process. See \cite[Subsection 12.1.1]{nguyen2025hybrid}.
\begin{asp}\label{A12.1}
	Assume the following hold.
	\begin{enumerate}
		\item $f_i(\cdot, \alpha), g_i(\cdot, \alpha):\R^n_+\to\R$ are locally Lipschitz functions for any $ \alpha\in\mathcal S.$
		\item
		There are $\delta_0\in(0,1), \ell >0, H>0$ and a  function $W :  \R^n_+\mapsto[1,\infty)$ such that
		\begin{equation}\label{c14-bound-W}
			\text{$W (\bx)$ is twice differentiable in $\bx$ and }
			\liminf_{\bx\to\infty} \frac{W (\bx)}{\ln \|\bx\|}>n.
		\end{equation}
		\begin{equation}\label{c14-a.tight}
			[\op W ](\bz)+2\delta_0 [\Gamma W](\bz )+\sum_{i=1}^n (|f_i(\bz )|+\sigma_{ii}^2g_i^2(\bz)\leq K\1_{\{\|\bx\|\leq \ell\}}-2
		\end{equation}
		where
		\begin{equation}\label{c14-hat-Lambda}
			[\Gamma W](\bz ):=	\sum_{i,j} |W _{x_i}(\bx ) x_iW _{x_j}(\bx ) x_jg_i(\bz )g_j(\bz )\sigma_{ij}\geq 0 \text{ forall } \bz\in\R^n_+\times\mathcal S.
		\end{equation}
	\end{enumerate}
\end{asp}
Let $I\subset \{1,\cdots,n\}$ and $I^c=\{1,\cdots,n\}\setminus I$.
Let
$$\R_+^{I}:=\{(x_1,\dots,x_n)\in\R^n_+: x_i=0\text{ if } i\in I^c\text{ and }x_i\geq 0\text{ if  }x_i\in I \}$$
$$\R_+^{I,\circ}:=\{(x_1,\dots,x_n)\in\R^n_+: x_i=0\text{ if } i\in I^c\text{ and }x_i>0\text{ if  }x_i\in I \}$$ and $\partial\R_+^{I }:=\R_+^I\setminus\R_+^{I ,\circ}$
We will apply Theorem \ref{thm5} with $\M=\R^n_+\times \mathcal S$, $\M_+=\R^{n,\circ}_+\times \mathcal S$ and $\M_0=\R_+^{I}\times \mathcal S$.

\begin{asp}\label{A12.3}
There exists  a $\rho>0$ such that
	\begin{equation}\label{c14-ae3.1}
		\int_{\R_+^{(n-k)}\times\mathcal S}\left(f_i\bz )-\dfrac{\sigma_{ii}g_i\bz )^2}{2}\right)\mu(d\bz) <-\rho \text{ for any } i\in I^c, \mu\in \PE(\R^{I,\circ}_+\times\mathcal S).
	\end{equation}
	Suppose further that
	for any $\nu\in\PE\left(\partial\R_+^{I }\times\mathcal S\right)$, we have
	\begin{equation}\label{c14-ae3.2}
		\max_{i\in I}\int_{\R_+^{I}\times\mathcal S}\left(f_i\bz )-\dfrac{\sigma_{ii}g_i\bz )^2}{2}\right)\nu(d\bz )>0
	\end{equation}
\end{asp}
Lemma 12.20 in \cite{nguyen2025hybrid} shows that $X_j(t),\; j \in I^c,$ converges to $0$ exponentially fast with high probability when $\sum_{j \in I^c} X_j(0)$ is sufficiently small. The proof relies on Assumptions \ref{A12.1}, \ref{A12.3}, and an additional condition (see \cite[{\bf (H12.4)}]{nguyen2025hybrid}), which is technicall more restrictive than \eqref{c14-a.tight}.

We show here that this additional condition is in fact not necessary.

\begin{thm}\label{thm3.1}
	Under Assumptions \ref{A12.1} and \ref{A12.3},
	for any $\bz=(\bx,s)\in\M_0:=\R^{I ,\circ}_+\times\mathcal S$ and $\eps>0$, there exists $\delta=\delta(\bz,\eps)>0$ such that
	$$
	\PP_{(\wdt\bx,\wdt s)}\left\{\lim_{t\to\infty}\frac{\ln X_i(t)}t=\int_{\R_+^{I}\times\mathcal S}\left(f_i\bz )-\dfrac{\sigma_{ii}g_i\bz )^2}{2}\right)\nu(d\bz )<0, i\in I^c_\mu\right\}>1-\eps \text{ if } |\wdt\bx-\bx|<\delta, \wdt s\in\mathcal S.
	$$
\end{thm}
\begin{proof}
	
In view of \cite[Lemma 4]{schreiber2011persistence}, the condition \eqref{c14-ae3.2} is equivalent to the existence of
$0<\hat p_i<1, i>k$
such that for any $\nu\in \PI\left(\partial\R_+^{I }\times\mathcal S\right),$  we have
$$\sum_{i=k+1}^n\hat p_i\int_{\R_+^{I}\times\mathcal S}\left(f_i\bz )-\dfrac{\sigma_{ii}g_i\bz )^2}{2}\right)\nu(d\bz )>0$$
Thus, there is $\check p\in (0,1)$ sufficiently small such that
\begin{equation}\label{c14-e3.2}
	\begin{aligned}
		\sum_{i=k+1}^n&\hat p_i\int_{\R_+^{I}\times\mathcal S}\left(f_i\bz )-\dfrac{\sigma_{ii}g_i\bz )^2}{2}\right)\nu(d\bz )\\
		&-\check p\max_{i\leq k} \int_{\R_+^{I}\times\mathcal S}\left(f_i\bz )-\dfrac{\sigma_{ii}g_i\bz )^2}{2}\right)\nu(d\bz )>0 \text{ for any }\nu\in \PI\left(\partial\R_+^{I }\times\mathcal S\right)
	\end{aligned}
\end{equation}
We define:
	$$U_j(\bx )=-W(\bx)+\sum_{i\in I } \hat p_i\ln( x_i)-\check{p}\ln  x_j, j\in I^c_\mu \text{ and } U(\bx)= \min_{j\in I^c_\mu} U_j(\bx).$$
	We have $$\lim_{u\to 0} \sup_{\{\bx: U(\bx)\geq u\}} \dist(\bx, \M_0)=0.$$
We also see that  $\op U_j(\bz)=H_j(\bz), \bz\in\M_+, j\in I^c$, where \begin{equation}\label{c14-e3.4}H_j(\bz)=-\op W(\bx)+\sum_{i\in I }\hat p_i\left(f_i(\bz )-\dfrac{\sigma_{ii}g_i(\bz )^2}{2}\right) -\check p \left(f_j(\bz )-\dfrac{\sigma_{jj}g_j\bz )^2}{2}\right), j\in I^c\end{equation}
is a continuous function on $\M$.
	Due to \eqref{c14-bound-W}, we can verify that $\op U\geq 1$ if $\|\bx\|\geq\ell$. 
In view of \cite[Lemma 12.4]{nguyen2025hybrid}, $\mu\in \PI(\R_+^{n}\times\mathcal S),$
$\mu \op W=0$.
As a result, applying  \eqref{c14-e3.2} to \eqref{c14-e3.4}
yields
\begin{equation}\label{c14-e3.5}
	\mu H_j \geq \lambda \text{ for  any }\mu\in \PI(\M_0), j\in I^c_\mu.
\end{equation}
We can compute that
\begin{align*}
[\Gamma U_i](\bz)=& -\sum_{i,j} \left((W _{x_i}(\bx ) x_iW _{x_j}(\bx ) x_jg_i(\bz )g_j(\bz )\sigma_{ij}\right)\\
& +\sum_{i\in I }\hat p_i\sigma_{ii}g_i^2(\bz )-\check p \sigma_{jj}^2g_j(\bz )
\end{align*}
We further obtain from \eqref{c14-a.tight} that
$$
[\op W]\leq K-\delta \Gamma U_i 
$$
which easily imply that
$$
\int_0^t P_{\bz}[\Gamma U_i](\BZ(t))dt\leq W(\bx) +Kt, t\geq 0.
$$
Since Assumption \ref{asp5} has been verified, we can apply Theorem \ref{thm5}. It follows that, under Assumptions \ref{A12.1} and \ref{A12.3}, for any $\eps>0$ and $\bar m>0$, there exists $u_{\eps,\bar m}>0$ such that
\[
\PP_\bz\left\{\liminf_{t\to\infty} \frac{U(X(t))}{t}\geq \lambda_0>0\right\}\geq 1-\eps,
\quad \forall \bz\in \M_+:\; U(\bx)\geq u_{\eps,\bar m},\ \|\bx\|\leq \bar m.
\]

Note that $\liminf_{t\to\infty} \frac{U(X(t))}{t}<0$ implies $\lim_{t\to\infty} X_j(t)=0$ for all $j\in I^c$. 

Then, by analyzing the associated random occupation measures and using arguments similar to those in the proof of \cite[Lemma 12.20]{nguyen2025hybrid}, we can show that for almost all $\omega$ in
\[
\left\{\omega:\; \liminf_{t\to\infty} \frac{U(X(t))}{t}\leq \lambda_0\right\},
\]
we have
\[
\lim_{t\to\infty}\frac{\ln X_i(t)}{t}
= \int_{\R_+^{I}\times\mathcal S}
\left(f_i(\bz)-\frac{(\sigma_{ii} g_i(\bz))^2}{2}\right)\nu(d\bz),
\]
which completes the proof.
\end{proof}

\begin{rem}
Theorem \ref{thm3.1} shows a  convergence in probability to an "attracting" subspace $R^{I,\circ}_+$ when the initial value is close. By a further condition for the accessibility of the boundary and a technical condition to exclude critical cases (\cite[{\bf (H12.5)}]{nguyen2025hybrid}), we can show that the solution to \eqref{kol} with positive inital will converge to the boundary $\partial\R^n_+$ with probability 1. We refer to \cite[Theorem 12.15]{nguyen2025hybrid} for details.
\end{rem}

\subsection{Disease-free state stability in an SIRS model with regime switching}
In this subsection, we present a couple of examples where the quadratic variation of a feasible Lyapunov function $U$ is not linearly bounded.
First, consider a SIRS model 
\begin{equation}\label{e1.1}
	\begin{cases}
		dS(t)=\left(b(\alpha(t))- I(t)F(S(t),I(t),\alpha(t))-c_1(\alpha(t)) S(t)+\gamma(\alpha(t)) R(t)\right)dt+\sigma_1(\alpha(t)) S(t)dE_1(t)  \\
		dI(t)=\left(I(t)F(S(t),I(t),\alpha(t))-c_2(\alpha(t)) I(t)\right)dt+\sigma_2(\alpha(t)) I(t)dE_2(t)  \\
		dR(t)=(c_4(\alpha(t)) I(t) - c_3(\alpha(t)) R(t))dt+\sigma_3(\alpha(t)) R(t)dE_3(t).
	\end{cases}
\end{equation}
where $b, c_i, i=1,\cdots, 4$ are positive functions $\mathcal S\to\R$ and
$\sigma_j,j=1,2,3$ are functions $\mathcal S\to\R$, $E_i, i=1,2,3$ are three independent Brownian motions.
We assume that $F$ is a locally Lipschitz function and linearly bounded. 
To see that it is not practically feasible to apply the main result in \cite{foldes2024stochastic} or our Theorem \ref{thm5} to obtain a sharp condition for extinction of the disease $I(t)$, we examine possible choices of Lyapunov functions satisfying Assumption \ref{asp-FS}.

Since we aim to establish an exponential rate of convergence of $I(t)$ to $0$, it is natural to consider a function $V$ involving $-\ln I(t)$. However, the generator $\op[-\ln I(t)]$ contains a term involving $S(t)$, which necessitates introducing an auxiliary function $W$ such that $\op W(X(t))$ includes a negative term that dominates the growth of $S(t)$.

Based on the structure of \eqref{e1.1}, a natural candidate is $W(s,i,r,\alpha)=s+i+r$. However, for this choice, the associated quadratic variation is only linearly bounded when $\sigma_1(\alpha)$ is sufficiently small. A similar difficulty arises when attempting to apply Theorem \ref{thm5}. However, we will show that Theorem \ref{thm6} is applicable here.

To proceed, we examine when $I(t)=R(t)=0$ and $S(t)$ follows the equation
	\begin{equation}\label{e1.S}
		dS(t)=\left(b(\alpha(t))-c_1(\alpha(t)) S(t)\right)dt+\sigma_1(\alpha(t)) S(t)dE_1(t).
\end{equation}

It is well known (see e.g. \cite[Theorem 2.2]{nguyen2020general}) that $(S(t),\alpha(t))$ satisfying \eqref{e1.S} and \eqref{c14-eq:tran} has a unique invariant measure on $[0,\infty)\times\mathcal S$ that we denote by $\pi$.
\begin{thm}\label{thm2e}
Suppose 
	\begin{equation}\label{e3.22}
		\lambda_I= -\int_{[0,\infty)\times\mathcal S} \left(f(s, 0,\alpha)-c_2(\alpha)-\frac{\sigma_2^2}2\right)\pi(ds,d\alpha)>0.
	\end{equation}
Then, $$\PP_{s, i, r, \alpha}\left\{\lim_{t\to\infty} \frac{\ln I(t)}t=-\lambda_I<0, \lim_{t\to\infty}\frac{\ln R(t)}t=-\lambda_R\wedge\lambda_I<0\right\}=1, \text{ for all } (s,i, r,\alpha)\in(0,\infty)^3\times\mathcal S.$$
where
	$$\lambda_R=\sum_{\alpha\in\mathcal S}\left(c_3(\alpha)+\frac{\sigma_3^2(\alpha)}2\right)\pi(\{\alpha\})$$
\end{thm}
\begin{proof}
Using the transformation: $X(t)=I(t)+R(t), Y(t)=I(t)/X(t)$, \eqref{e1.1} with a positive initial value will become
	\begin{equation}\label{e3.2}
	\begin{cases}
		dS(t)=&\left(b(\alpha(t))- Y(t)X(t)F(S(t),Y(t)X(t),\alpha(t))-c_1(\alpha(t)) S(t)+\gamma(\alpha(t)) (1-Y(t))X(t)\right)dt\\&+\sigma_1(\alpha(t)) S(t)dE_1(t)  \\
		dX(t)=&X(t)\left(Y(t)[F(S(t),Y(t)X(t),\alpha(t))-c_2(\alpha(t)) +(c_4(\alpha(t))]-c_3(\alpha(t)) (1-Y(t))\right)dt\\&+\sigma_2(\alpha(t)) Y(t)X(t)dE_2(t)  
+\sigma_3(\alpha(t)) R(t)dE_3(t).\\
dY(t)=&Y(t)f_Y(S(t), I(t), Y(t))dt +Y(t)(1-Y(t)) (\sigma_2 dE_2 - \sigma_3 dE_3)
	\end{cases}
\end{equation}
The solution process $\BZ(t)=(S(t), X(t), Y(t),\alpha(t))$ stays in  $\M_+:=[0,\infty)\times(0,\infty)\times[0,1]\times\mathcal S$ for any initial value $\bz=(s, x, y, \alpha)\in\M_+$.
The process can be continuously extended as a Markov Feller process on $\M:=[0,\infty)\times(0,\infty)\times[0,1]\times\mathcal S.$
Consider function $W(\bz):=(s+x)^{1+2p_0}$.
We can easily check that for $p_0>0$ be sufficiently small, $W$ satisfies conditions (1), (2) and (4) of Theorem \ref{thm6}.
where
When $X(t)=0$, the process lives in $\M_0=[0,\infty)\times\{0\}\times[0,1]\times\mathcal S$ satisfying
\begin{equation}\label{e3.3}
	\begin{cases}
		dS(t)=&\left(b(\alpha(t))- c_1(\alpha(t)) S(t)\right)dt+\sigma_1(\alpha(t)) S(t)dE_1(t)  \\
		dY(t)=&Y(t)f_Y(S(t), 0)dt +Y(t)(1-Y(t)) (\sigma_2 dE_2 - \sigma_3 dE_3)
	\end{cases}
\end{equation}
where
$$f_Y(\bz)=[(1-y)(F(s,xy,\alpha)-c_2(\alpha)+c_3)-c_4y+\sigma_3^2(\alpha) (1-y)^2 -\sigma_2^2(\alpha)y(1-y)]$$
Because the equation of $S(t)$ in \eqref{e3.3} is independent of $Y(t)$, any ergodic measure of \eqref{e3.3} has the form
\begin{equation}\label{e3.21}
\mu(ds,dy,dk)=\pi(ds,dk) \rho (dy|ds,dy)
\end{equation}
where $\pi(ds,dk)$ is the unique invariant probability measure of $(S(t),\alpha(t))$ which satisfies \eqref{e1.S}.

Note that $yf_y(s, 0, y, \alpha)=-c_4(\alpha)<0$ when $y=1$.
Thus $y=1$ is the entrance point of $Y(t)$ on $[0,1]$.
On the other hand, $y=0$ is a natural boundary.
As such, for any ergodic measure $\mu$, we have $\mu(\{0<y<1\})=1$ or $\mu(\{y=0\})=1$.

If $Y(t)>0$, we have
$$
\ln Y(t)= f_Y(S(t), 0, Y(t),\alpha(t))dt +\frac{(1-Y(t))^2(\sigma_2^2+\sigma_3^2)}2 dt + (1-Y(t))(\sigma_2 dE_2 - \sigma_3 dE_3)
$$
Thus, if $\mu(\{0<y<1\})=1$, we must have
\begin{equation}\label{e3.23}
\mu(\wdt f_Y(s, 0, y,\alpha)=0
\end{equation}

$$\wdt H(s,0, y,\alpha)= \wdt f_Y(s,0,y,\alpha)-\left(F(s, 0, \alpha) -c_2(\alpha) +\frac{\sigma_2^2}2\right)$$

In view of \eqref{e3.21}, \eqref{e3.22} and \eqref{e3.23},
we have
\begin{equation}\label{e3.24}
	\mu(\wdt  H(s,0, y,\alpha))=\lambda>0 \text{ if } \mu\in\PI(\M_0), \mu(\{0<y<1\})=1
\end{equation}
If $\mu(\{y=0\})=1$, we can easily have $H(s,0,0,\alpha)=c_3(\alpha)+\frac{\sigma_3^2}2$
\begin{equation}\label{e3.25}
	\mu(\wdt H(s,0, y,\alpha))=\sum_{\alpha\in\mathcal S}\left(c_3(\alpha)+\frac{\sigma_3^2(\alpha)}2\right)\nu_\alpha=:\wdt\lambda>0 \text{ if } \mu\in\PI(\M_0), \mu(\{y=0\})=1
\end{equation}
By  \cite[Remark 13]{benaim2018stochastic}
\begin{equation}\label{e3.26}
	\mu(\op U_1)=0
\end{equation}
For $U(\bz)=\ln (1+s+x)-\ln x=U_1(s, x)-\ln x $, we can see that 
$$[\op U](\bz)=\op U_1 -\wdt H(\bz)=:H(\bz) \text{ on } \M_+$$
where $H(\bz)$ is a continuous function on $\M$. We also see that
$$
H(\bz)\leq c_1 - c_2 (s+x) \text{ for some } c_1,c_2>0, \bz\in\M.
$$
As a result, condition (3) of Assumption \ref{asp4} holds for the proposed triple $(W, U, H)$.

Combining \eqref{e3.23}, \eqref{e3.24}, \eqref{e3.25} and \eqref{e3.26}, we have
$$
\mu H\geq \min\{\lambda, \wdt\lambda\}>0 \text{ for any } \mu\in\PI(\M_0)
$$
which means Assumption \ref{asp2} 
Since we have verified that $(W, U, H)$ satisfies Assumptions \ref{asp2} and \eqref{asp4}, we derive from Theorem \ref{thm2} that for any $\eps>0, K>0$, there exists a $\delta>0$ satisfying
\begin{equation}\label{e3.27}
\PP_{\bz}\left\{\liminf_{t\to\infty}\frac{\ln [(1+S(t)+X(t))/X(t)]}t \leq -\lambda_0<0\right\} \text{ for all } \bz\in\M: s+x\leq K, x\leq \delta
\end{equation}
where $\lambda_0$ is a nonrandom constant independent of $\eps,K$.
Because the hybrid diffusion \eqref{e1.1} is nondegenerate, \eqref{e3.27} implies that the hybrid diffusion is transient (see e.g. \cite[Section 3.3]{nguyen2025hybrid}).
Then using arguments about the weak limits of  random occupation measures as those in the proofs of \cite[Proposition 4.1 \& Theorem 2.2]{nguyen2020long}, we can easily obtain 
\end{proof}
\subsection{Extinction of one predator in a two-predator–one-prey model}

Here is another example where the linear boundedness of the quadratic variation of a Lyapunov function is difficult to verify. We will apply Theorem \ref{thm6} instead.
We consider a stochastic Lotka--Volterra model describing a food web with two predators and one prey:
\begin{equation}\label{e3.31}
	\begin{cases}
		dX_1(t)=X_1(t)\big(r_1-a_{11}X_1(t)-a_{12}X_2(t)-a_{13}X_3(t)\big)dt+\sigma_1X_1(t)\,dE_1(t),\\
		dX_2(t)=X_2(t)\big(-r_2+a_{21}X_1(t)-a_{23}X_3(t)\big)dt+\sigma_2X_2(t)\,dE_2(t),\\
		dX_3(t)=X_3(t)\big(-r_3+a_{31}X_1(t)-a_{32}X_2(t)\big)dt+\sigma_3X_3(t)\,dE_3(t).
	\end{cases}
\end{equation}

where $X_1(t)$ denotes the prey population size, while $X_2(t)$ and $X_3(t)$ denote the two predator populations. The constants $r_i>0$ represent intrinsic growth or death rates, and $a_{ij}>0$ describe interaction coefficients between species. The processes $E_i(t)$, $i=1,2,3$, are independent standard Brownian motions modeling environmental noise, and $\sigma_i$ are the noise intensities.
We denote $\BX(t)=(X_1(t),X_2(t), X_3(t))\in \M:=[0,\infty)$ be the solution process with initial value $\bx=(x_1,x_2,x_3)$. Define $\M_+=(0,\infty)^3$ and $\M_0=[0,\infty)^2\times\{0\}$.
In view of \cite{hening2018coexistence}, we know that
when $X_2(t)=X_3(t)=0$, $\{\BX(t)\}$ has an unique invariant probability measure on $(0,\infty)\times\{0\}^2$, denoted by $\mu_1$ if $r_1-\frac{\sigma_1^2}2>0$.
\begin{equation}\label{a1-sub12-c1}
	\frac12\sigma_1^2<r_1 
\end{equation}
Moreover, under that condition, 
 the system on $(0,\infty)^2\times\{0\}$ (i.e. $X_3(t)=0$)
has a unique invariant probability measure $\mu_{12}$ if

\begin{equation}\label{a1-sub12-c2}
 \left(r_1-\frac12\sigma_1^2\right)\frac{a_{21}}{a_{11}}-r_2-\frac{\sigma_2^2}2>0.
\end{equation}
Furthermore, 
$$
\E_{\mu_{12}}(X_1):=\int_{(0,\infty)^2} x_1\mu_{12}(dx_1dx_2)= \frac{r_2+\frac{\sigma_2^2}2}{a_{21}}
$$ and $$ \E_{\mu_{12}}(X_2):= \int_{(0,\infty)^2} x_2\mu_{12}(dx_1dx_2)= \frac1{a_{12}}\left[r_1-\frac12\sigma_1^2-\frac{r_2+a_{11}\frac{\sigma_2^2}2}{a_{21}}\right]
$$

We will show that, under \eqref{a1-sub12-c1} and \eqref{a1-sub12-c2}, a condition for extinction of $X_3$ in \eqref{e3.31} is
$$
\lambda_3(\mu_{12}):=-r_3-\frac{\sigma_3^2}2 +a_{31}\E_{\mu_{12}} X_1 -a_{32}\E_{\mu_{12}} X_2<0.
$$

We have 
\begin{equation}
\begin{cases}
	\op[\ln x_1]=&r_1-a_{11}y x_1-a_{12}y x_2-a_{13}y x_3 -\frac{\sigma_1^2}2=: H_1(\bx),\\
\op[\ln x_2]=&	-r_2+a_{21}y x_1-a_{23}y x_3-\frac{\sigma_2^2}2=: H_2(\bx)\\
\op[\ln x_3]=&	-r_3+a_{31}y x_1-a_{32}y x_2-\frac{\sigma_3^2}2=: H_3(\bx)
\end{cases}	
\end{equation}

On $\M_0=[0,\infty)^2\times\{0\}$, there are 3 ergodic measures $\bdelta$, which is the Dirac measure at the original, $\mu_1$, the ergodic measure on the $x_1$ axis and $\mu_{12}$.
We have $\delta_1(H_1)>0$, $\mu_1(H_2)>0$ and $\mu_{12}(H_3)<0$.
On the other hand, we have $\mu_1(H_1)=\mu_{12}(H_1)=\mu_{12}(H_2)=0$ (see \cite[Lemma 5.1]{hening2018coexistence} or \cite[Remark 13]{benaim2018stochastic}).
Thus, let $p_1\in(0,1)$ and $p_2\in(0,1)$ sufficiently small relative to $p_1$ and $p_3\in(0,1)$ sufficiently small relative to $p_2$, we have
$\bdelta(p_1H_1+p_2H_2-p_3H_3)>0$ as well as
\begin{eqcite}
	\bdelta(p_1H_1+p_2H_2-p_3H_3)=&p_1\bdelta(H_1)+p_2\bdelta(H_2)+p_3\bdelta(H_3)>0\\
\mu_{12}(p_1H_1+p_2H_2-p_3H_3)=&p_2\mu_1(H_2)-p_3\mu_{12}(H_3)>0\\
\mu_{12}(p_1H_1+p_2H_2-p_3H_3)=&-p_3\mu_{12}(H_3)>0
\end{eqcite}
which implies the existence of $\Lambda>0$ satisfying
\begin{equation}\label{e3.36}
\mu(p_1H_1+p_2H_2-p_3H_3)\geq \Lambda \text{ for all } \mu\in \PI(\M_0) 
\end{equation}
Let $b_0=\frac12\left(\frac{a_{12}}{a_{21}}\wedge\frac{a_{12}}{a_{21}}\right)$ and $r_0=b_0(r_2\wedge r_3)$
we can easily check that
$$
\op W_0\leq K_0 -r_0W_0 \text{ and } \Gamma W_0 \leq k_0 W_0\text{ for all } \bx\in[0,\infty)^3,
$$ 
where $W_0(\bx)=1+x_1+b_0x_2+b_0x_3$ and $K_0, k_0$ are some positive constants.
Moreover, with $p_0>0$ be sufficiently small and some positive constant $K_1,k_1$, we have
\begin{equation}\label{e3.37}
	\op W^{1+2p_0}_0\leq K_1 - \frac{r_0}2 W^{1+2p_0}_0  \text{ and } \Gamma W^{1+2p_0}_0 \leq k_1 W^{1+2p_0}_0\text{ for all } \bx\in[0,\infty)^3
\end{equation}
We can also find a $K^\diamond>0$ and a $\gamma^\diamond>0$ such that 
$$
\op \ln W_0 \leq K^\diamond - \gamma^\diamond (x_1+x_2+x_3) \text{ for all } \bx\in(0,\infty)^3.
$$

Thus, for $U=-C_0\ln W_0 +p_1\ln x_1 + p_2\ln x_2 -p_3\ln x_3$,
if $C$ is sufficiently large, we can find $K_U>0, \gamma_U>0$ satisfying
\begin{equation}\label{e3.38}
H:= -C_0[\op\ln W_0] + p_1H_1+p_2H_2-p_3H_3\geq \gamma_U \|\bx\| - K_U \text{ for all }\bx\in\M.
\end{equation}

\eqref{e3.37} and \eqref{e3.38} verify that Assumption \ref{asp4} is satisfied for $W=W^{1+2p_0}_0$, $U=-C_0\ln W_0 +p_1\ln x_1+p_2\ln x_2-p_3\ln x_3$ and $H$ defined above.

Since $\mu(\op\ln W_0)=0$ for any $\mu\in\PI(\M_0)$ due to \cite[Remark 13]{benaim2018stochastic} again, we have from \eqref{e3.36} that
$\mu(U)\geq\Lambda$ for all $\mu\in\PI(\M_0)$.
Thus, Assumption \ref{asp2} is satisfied.

As a result, the conclusion of Theorem \ref{thm2} holds for $U$, which shows $X_3(t)$ tends to $0$ with a large probability when the initial condition is sufficiently close to $\M_0$. Then using the weak limit of the random occupation  measures, we can show exactly the rate of convergence:
	$$
\PP_{\bx}\left\{\lim_{t\to\infty}\frac{\ln X_3(t)}t=\lambda_3(\mu_{12})<0\right\}>1-\eps \text{ if } \frac{(1+x_1+b_0x_2+b_0x_3)^{C_0}x_3^{p_3}}{x_1^{p_1}x_2^{p_2}}<\delta.
$$
for sufficiently small $\delta=\delta(\eps)$.

%
%

\bibliographystyle{amsalpha}
\bibliography{Kolmogorov}
\appendix
\section{Proofs}
In this apprendix, 	let $\{X_n\}$ be a sequence of random variables  and $\F_n$ adapted. Suppose
$$M_p:=\sup_{n}\E |X_n|^p <\infty \text{ for some }  p>1.$$
We write $X_n=Y_n+Z_n$ where $Y_n=X_n\1_{\{|X_n|\leq n\}}$, $Z_n=X_n\1_{\{|X_n|> n\}}$
\begin{lm}\label{lm-a3}
	For $p>1$, we have
	$\lim_{k\to\infty} \sum_{n=1}^\infty \frac1{(k+n)^p}=0.$
\end{lm}
\begin{proof}
	Since $$ \frac1{(k+n)^p}\leq \int_{n-1}^n\frac1{(k+x)^p}dx\, \text{ for } k>0, n\geq 1,$$ we have
	$$\sum_{n=1}^\infty \frac1{(k+n)^p}\leq \int_0^\infty \frac1{(k+x)^p}dx\leq \frac{p-1}{k^{p-1}}\to 0 \text{ as } k\to\infty.$$
\end{proof}
\begin{lm}\label{lm-a1} Let $\{a_i\}$ be a sequence satisfying
	$\left|\sum_{i=n_0}^n a_i\right|\leq \eps$ for any $n\geq n_0$ then
	$$
	\left|\sum_{i=n_0}^n ia_i \right|\leq 2\eps n,\,\forall\, n\geq n_0.
	$$
\end{lm}
\begin{proof}
	Let $A_i=\sum_{k=n_0}^{i} a_k, i\geq n_0$ and $A_{n_0-1}=0$.
	Summing by parts, we have
	$$\left|\sum_{i=n_0}^n ia_i\right| =\left|\sum_{i=n_0}^n i (A_i-A_{i-1})\right|=\left|(n+1) A_{n} - \sum_{i=n_0}^{n} A_{i}\right|\leq 2(n+1-n_0)\eps$$
\end{proof}
\begin{lm}\label{lm-a2}
	For any $n_0\in\N$ satisfying $M_p\sum_{n=n_0}^\infty n^{-p}\leq \eps\delta^2$, we have
	$$\PP\left(\left|\sum_{i=n_0}^k \left(Y_i-\E(Y_i\big|\F_{i-1})\right)\right|\leq 2k\delta \text{ for all } k\geq n_0\right)\geq1-\eps.$$
\end{lm}
\begin{proof}
	\begin{equation}\label{a2-e1}
	\begin{aligned}
		\sum_{n=n_0}^\infty n^{-2}\E\left[Y_n-\E(Y_n\big|\F_{n-1})\right]^2\leq & \sum_{n=n_0}^\infty n^{-2}\E Y_n^2\\
		\leq & \sum_{n=n_0}^\infty n^{-2} n^{2-p}\E |X_n|^p\\
		\leq & M_p\sum_{n=n_0}^\infty n^{-p}\leq \eps\delta^2
	\end{aligned}
	\end{equation}
	By Doob's inequality
	\begin{align*}
		\PP\left(\max_{n_0\leq k\leq n}\left|\sum_{i=n_0}^k i^{-1}\left(Y_n-\E(Y_n\big|\F_{n-1})\right)\right|\geq \delta\right)
		\leq& \frac1{\delta^2} \E \left(\sum_{i=n_0}^n i^{-1}\left(Y_i-\E(Y_i\big|\F_{i-1})\right)\right)^2\\= &
		\frac1{\delta^2}	\sum_{i=n_0}^n i^{-2}\E\left[Y_i-\E(Y_i\big|\F_{i-1})\right]^2\leq\eps
	\end{align*}
	which is true for any $n\geq n_0$ because of \eqref{a2-e1} and the fact that
	$$\E \left[\left(Y_i-\E(Y_i\big|\F_{i-1})\right)\left(Y_j-\E(Y_j\big|\F_{j-1})\right)\right]=0 \text{ if } i\ne j.$$
	As a result,
	$$\PP\left(\left|\sum_{i=n_0}^k i^{-1}\left(Y_i-\E(Y_i\big|\F_{i-1})\right)\right|\geq \delta \text{ for all } k\geq n_0\right)\leq\eps.$$
	In view of Lemma \ref{lm-a1},
	$$\PP\left(\left|\sum_{i=n_0}^k \left(Y_i-\E(Y_i\big|\F_{i-1})\right)\right|\leq 2k\delta \text{ for all } k\geq n_0\right)\geq1-\eps.$$
	
\end{proof}
\begin{lm}\label{lm-a4}
	For any $n_0$ satisfying $2M_p\sum_{n=n_0}^\infty n^{-p}\leq \eps\delta$, we have \begin{equation}\label{a4-e1}
\PP\left(\sum_{i=n_0}^n \left|Z_i-\E(Z_i\big|\F_{i-1})\right|\leq 2n\delta \text{ for all } k\geq n_0\right)\geq1-\eps\end{equation}
\end{lm}
\begin{proof}
	\begin{align*}	\E\left|Z_n-\E\left(Z_n\big|F_{n-1}\right)\right|\leq& \E |Z_n|+\E\left|\E\left(Z_n\big|F_{n-1}\right)\right|\leq 2\E|Z_n|\\
		\leq & 2\E |X_n\1_{\{|X_n|>n\}}|\\
		\leq & 2 n^{1-p}\E |X_n|^p\leq 2M_pn^{1-p}
	\end{align*}
	By Markov's inequality
	\begin{align*}
		\PP\left(\sum_{i=n_0}^n i^{-1}\left|Z_i-\E(Z_i\big|\F_{i-1})\right|\geq \delta\right)
		\leq& \frac1{\delta} \E \sum_{i=n_0}^n i^{-1}\left|Z_n-\E(Z_n\big|\F_{n-1})\right|\\\leq &
		\frac1{\delta}	M_p\sum_{i=n_0}^n n^{-p}\leq\eps
	\end{align*}
	Thus, we apply Lemma \ref{lm-a1} to obtain
	$$\PP\left(\sum_{i=n_0}^n \left|Z_i-\E(Z_i\big|\F_{i-1})\right|\leq 2n\delta \text{ for all } k\geq n_0\right)\geq1-\eps.$$
\end{proof}
Now, we can prove Proposition \ref{pron1}.
\begin{proof}[Proof of Proposition \ref{pron1}]
	Note that
	$$X_n-\E\left(X_n\big|F_{n-1}\right)=Y_n-\E\left(Y_n\big|F_{n-1}\right)+Z_n-\E\left(Z_n\big|F_{n-1}\right).$$
	Let $n_0$ satisfy $M_p\sum_{n=n_0}^\infty n^{-p}\leq \min\{\eps\delta^2,0.5\eps\delta, \eps\}$, we have from Lemma \ref{lm-a2} and \ref{lm-a4} that
	\begin{equation}\label{pron1-ep1}\PP\left(\left|\sum_{i=n_0}^k \left(X_i-\E(X_i\big|\F_{i-1})\right)\right|\leq 4k\delta \text{ for all } k\geq n_0\right)\geq1-2\eps.\end{equation}
	By Markov inequality, for $m=2n_0M_p^{1/p}\eps$, we have
	\begin{equation}\label{pron1-ep2}
	\begin{aligned}
	\PP\left(\left|\sum_{i=1}^{n_0-1} \left(X_i-\E(X_i\big|\F_{i-1})\right)\right|\leq m \right)\leq& \frac1m\sum_{i=1}^{n_0-1}\E \left|X_i-\E(X_i\big|\F_{i-1})\right|\\
	\leq& \frac2m\sum_{i=1}^{n_0-1}\E|X_i|\leq \frac{2n_0 M_p^{1/p}}m=\eps
	\end{aligned}
	\end{equation}
	Combining \eqref{pron1-ep1} and \eqref{pron1-ep2}, we have
	\begin{equation}\label{pron1-ep3}
	\begin{aligned}
	\PP\left(\left|\sum_{i=1}^{n} \left(X_i-\E(X_i\big|\F_{i-1})\right)\right|\leq m +4n\delta  \,\text{ for all } n\in\N\right)\geq 1-3\eps
	\end{aligned}
	\end{equation}
	which completes the proof.
\end{proof}

\begin{proof}[Proof of Lemma \ref{lm1-thm1}]
	For $A_{M,\eps}^R$,	it is implicitly proved in the book under (5) of Assumption \ref{asp2}. Alternatively, one can use Proposition \ref{pron1}.
	
	$\wdt \Delta_n$ is a bounded martingale difference w.r.t $\F_{nT_2}$, so the proof for $B_{M,\eps}^R$ is obvious.
\end{proof}

\end{document}